\documentclass{article}
\usepackage{extarrows,mathtools,amsmath,amssymb,epsfig,enumerate,hyperref,url,cases}
\usepackage{amsthm}
\usepackage{parskip}
\newtheorem{theorem}{Theorem}[section]
\newtheorem{proposition}[theorem]{Proposition}
\newtheorem{lemma}[theorem]{Lemma}
\newtheorem*{conjecture}{Conjecture}
\newtheorem{corollary}[theorem]{Corollary}
\theoremstyle{remark}
\newtheorem{remark}{Remark}[section]
\newtheorem{example}{Example}[section]
\newtheorem{counterexample}{Counterexample}[section]

\theoremstyle{definition}
\newtheorem{definition}{Definition}[section]

\DeclareMathOperator{\Aff}{Aff}

\DeclareMathOperator{\Dic}{Dic}

\usepackage{thmtools}
\usepackage{thm-restate}

\renewenvironment{abstract}
 {\par\noindent\textbf{\abstractname.}\ \ignorespaces}
 {\par\medskip}

\newenvironment{acknowledgments}{%
  \renewcommand{\abstractname}{Acknowledgments}
  \begin{abstract}
}{%
  \end{abstract}
}

\title{Almost and quasi-Leinster groups}
\author{Iulia - C\u at\u alina Ple\c sca, Marius T\u arn\u auceanu}

\begin{document}

\thispagestyle{plain}
\maketitle
\begin{center}{\bf Iulia - C\u at\u alina Ple\c sca, Marius T\u arn\u auceanu}\end{center}\vspace{3mm}

\begin{abstract}
In this paper, we study the parallelism between perfect numbers and Leinster groups and continue it by introducing the new concepts of almost and quasi-Leinster groups which parallel almost and quasi-perfect numbers. These are small deviations from perfect numbers; very few results and/or examples are known about them.

We investigate nilpotent almost/quasi-/Leinster groups and find some examples and conditions for the existence of such groups for classes of non-nilpotent groups: ZM (Zassenhaus metacyclic) groups, affine groups, dihedral groups and dicyclic groups.  
\end{abstract}

\section{Introduction}

Throughout this paper, all the integers are positive and all the groups are finite. We denote the set of positive integers by $\mathbb{N}^*$. By \textit{proper divisor} of an integer, we understand a divisor of the number that is different from the integer itself. Similarly, by a \textit{proper subgroup} of a group, we understand a subgroup of the group except the group itself.  
For a group $G$, we denote by $L(G)$ the lattice of its subgroups and by $N(G)$ the sublattice of normal subgroups. 

In 2001, Tom Leinster introduced in \cite{Leinster} a group theoretical analogue to perfect numbers: finite groups where the order of the group is equal to the sum of the orders of its proper normal subgroups. He called these groups perfect, but since this name was already used, these groups were later called Leinster groups \cite{Medts}. 

Finding examples of Leinster groups of odd order proved difficult, taking ten years. An example was produced by Fran\,cois Brunault in reply to Tom Leinster’s post on MathOverflow \cite{post}. The use of computational programs brought thousands of examples of Leinster groups, see \cite{Nieuwveld}. Additionally, the existence of Leinster groups among different classes of groups (e.g. generalized quaternion groups, semisimple groups) was studied, see \cite{Baishya1}. Also, Leinster groups of specific orders, products of primes, such as $p^2q^2$, $pqrs$, have been studied, \cite{Baishya1, Baishya2}.  

Our paper aims to further study Leinster groups and introduce two new related concepts: almost and quasi-Leinster groups, group-theoretical analogues to almost and quasi-perfect numbers. 

\subsection{Preliminaries}
We begin with the number theoretic concepts and then present their group-theoretic equivalents.

We start with two classical integer-valued functions (\cite{Medts}, Definition 2.1.):

\begin{equation}
D(n)=\sum_{d|n}d
\end{equation} 
and 
\begin{equation}
\delta(n) = \frac{D(n)}{n}\,,
\end{equation}where $n\in\mathbb{N}^*$.

Now, we can classify integers based on the values of $D$ and $\delta$: 
\begin{definition}(\cite{Guy}, Pages 71, 74)
\begin{enumerate}[(1)]
\item
An integer $n$ is called \textit{perfect} if it is equal to the sum of its proper divisors, i.e.:
\begin{equation}
\delta(n)=2
\end{equation}
\item
An integer is called \textit{abundant} if it is smaller than the sum of its proper divisors, i.e.:
\begin{equation}
\delta(n)>2.
\end{equation}
\item
An integer is called \textit{deficient} if it is greater than the sum of its proper divisors, i.e.:
\begin{equation}
\delta(n)<2.
\end{equation}
\item
An abundant integer is called \textit{quasi-perfect} if the sum of its proper divisors is greater than itself by one, i.e.:
\begin{equation}
D(n)=2n+1.
\end{equation}
\item
A deficient integer is called \textit{almost perfect} if the sum of its proper divisors is smaller than itself by one, i.e.:
\begin{equation}
D(n)=2n-1.
\end{equation}
\end{enumerate}
\end{definition}

\begin{example}[\cite{Guy}, Pages 71, 74]
\begin{enumerate}[(1)]
\item 
$6=1+2+3$ and $28=1+2+4+7+14$ are perfect numbers.
\item
The only known almost perfect numbers are of type $2^s-1$ with $s>0$. 
\item
There are no known quasi-perfect numbers.
\end{enumerate}
\end{example}
\begin{proposition}[\cite{Guy}, Page 71]\label{prop:perfect}
All even perfect numbers are of the form $f(r)=2^{r-1}(2^r-1)$, where $r\geq 2$ and $2^r-1$ is prime.
\end{proposition}

\begin{example}[\cite{Mersenne}]
There are 52 known perfect even numbers $P_i$, $i\in\{1,\dots, 52\}$. Using notation in Proposition \ref{prop:perfect}, the first seven are: 
$P_1=6=f(2)$, $P_2=28=f(3)$, $P_3=496=f(5)$, $P_4=8128=f(6)$, $P_5=33550336=f(13)$, $P_6=8589869056=f(17)$ and $P_7=137438691328=f(19)$.
\end{example}
Using Proposition \ref{prop:perfect}, we prove a small corollary that will be useful in our study.

\begin{corollary}\label{k=1}
Let $p$ be a prime, $P_i$ an even perfect number and $k$ a positive integer such that:
\begin{equation}\label{eq:perfect}
P_i+1=p^k.
\end{equation} 
Then $k=1$.
\end{corollary}
\begin{proof}
Since $P_i$ is an even perfect number, using \ref{prop:perfect}, equation \eqref{eq:perfect} becomes: 
\begin{equation}\label{eq:perfect:prime_power}
p^k-1=2^{r-1}(2^r-1),
\end{equation} where $2^r-1$ is a Mersenne prime $r\geq 2$. \\
We assume, by contradiction, that $k$ is even, i.e. $k=2\ell, \ell\in\mathbb{N}^*$. Equation \eqref{eq:perfect:prime_power} becomes:
$$(p^\ell-1)(p^\ell+1)=2^{r-1}(2^r-1).$$
Numbers $p^\ell-1$ and $p^\ell+1$ are even consecutive numbers. It follows that $r-1\geq 2\iff r\geq 3.$\\
Since $2^r-1$ is prime, it divides only one of the factors in the left-hand side (either $p^\ell-1$, or $p^\ell+1$). Since they are both even, it follows that $2^{r+1}-2$ is a divisor of one of the factors. 
Because $2^{r+1}-2>2^{r-2}$, and $p^\ell+1>p^\ell-1$, we get $p^\ell-1=2^{r-2}$ and $p^\ell+1=2^{r+1}-2$. Subtracting these, it follows that $2^{r+1}-4=2^{r-2}$, which yields a contradiction, hence $k$ is odd.\\
Now, we prove that $k$ can actually be just $1$.\\
Assume by contradiction, that $k>1$. In this case, equation \eqref{eq:perfect:prime_power} becomes:
$$2^{r-1}(2^r-1)=(p-1)(p^{k-1}+\dots+p+1)$$
Since $p^k$ is odd, it follows that $p$ is odd. Together with the fact that $k$ is odd, we get that $p^{k-1}+\dots+p+1$ is odd. Since $p-1$ is odd, it follows that $2^{r-1}|(p-1)$. Also, because $2^r-1$ is prime, it divides just one of the factors in the right-hand side. Both the factors in the right hand side are greater than $1$; it follows that $2^{r-1}=p-1$ and $2^r-1=p^{k-1}+\dots+p+1$. We get that:
$2p-3=p^{k-1}+\dots+p+1$, which is equivalent to $p^{k-1}+\dots+p-2p=-4$. Since $p$ is odd, this has no solutions.\\
This concludes the proof.
\end{proof}
\begin{example}\label{perfect+1}
Using GAP \cite{GAP}, we found $4$ solutions for \eqref{eq:perfect}: 
\begin{equation}\label{Mersenne prime}
P_i\in \{P_1, P_2, P_5, P_7\}.
\end{equation}
They were obtained by checking the first $39$ even perfect numbers.
\end{example}
While even perfect numbers are characterized, no odd perfect numbers are known:
\begin{conjecture}[\cite{Guy}, Page 71]
There are no odd perfect numbers.
\end{conjecture}

Now, we can present the analogues concepts for groups. We start with the definitions of the functions $D$ and $\delta$.

Most of our notation is standard and will usually not be repeated here. Elementary notions and results on groups can be found in \cite{Huppert} and \cite{Isaacs}.
\begin{definition}[\cite{Medts}, Definition 2.1.]\label{def:functions}
Given a group $G$, we define the following: 
\begin{equation}
D(G)=\sum_{N\triangleleft G}|N|
\end{equation} 
and 
\begin{equation}
\delta(G) = \frac{D(G)}{|G|}.
\end{equation}
\end{definition}
The easiest computations of the functions $D$ and $\delta$ are those for cyclic groups.
\begin{example}[\cite{Leinster}, Example 2.1.]\label{sigma:cyclic}
For a cyclic finite group $C_n$, the following hold:
\begin{equation}\label{D}
D(C_n)=D(n)
\end{equation} 
and
\begin{equation}\label{delta}
\delta(C_n)=\delta(n).
\end{equation} 
\end{example}
We recall some basic properties of the two functions given in Definition \ref{def:functions}.
\begin{remark}[\cite{Medts}, Observation 3.1.]
The following hold:
\begin{enumerate}[(1)]
\item
$\delta(G)>1$, for all $G\neq \{1\}$;
\item
$\delta(G/N)\leq \delta(G),$ for all $N \lhd G$;
\item
$D$, $\delta$ are multiplicative functions:
$$
\text{If }G_1, G_2 \text{ groups such that }\gcd(|G_1|,|G_2|)=1, \text{then }
$$
\begin{equation}
D(G_1\times G_2)=D(G_1)D(G_2)
\end{equation}
\begin{equation}
\delta(G_1\times G_2)=\delta(G_1)\delta(G_2)
\end{equation}
\end{enumerate}
\end{remark}
Now, we can state the definitions for Leinster, abundant and deficient groups. The first ones have been studied in \cite{Leinster,Medts,Baishya1,Baishya2,Nieuwveld}. \begin{definition}[\cite{Medts}, Definition 2.1.]\label{def:Leinster}
Let $G$ be a group.
\begin{enumerate}[(1)]
\item
$G$ is called \textit{Leinster} if $\delta(G)=2$.
\item
$G$ is called \textit{abundant} if $\delta(G)>2$.
\item
$G$ is called \textit{deficient} if $\delta(G)<2$.
\end{enumerate}
\end{definition}
We can extend the analogy to integers by introducing quasi-Leinster and almost Leinster groups.
\begin{definition}
Let $G$ be a group.
\begin{enumerate}[(1)]
\item
$G$ is called \textit{quasi-Leinster} if 
\begin{equation}
D(G)=2|G|+1.
\end{equation}
\item
$G$ is called \textit{almost Leinster} if 
\begin{equation}
D(G)=2|G|-1.
\end{equation}
\end{enumerate}
\end{definition}

\begin{counterexample}
The non-abelian group $G_{p,q}$ of order $pq$, $p<q$ primes such that $p|q-1$, is not almost/quasi-Leinster.
To see this, we compute the value of $D(G_{p,q})$. From Sylow's theorems, it follows that $G_{p,q}$ has as proper normal subgroups the trivial one and the Sylow one of order $q$. Therefore 
$$D(G_{p,q})=1+q+pq.$$Since $|G_{pq}|=pq$, we would have to have:
$$2pq\pm 1= 1+q+pq \iff pq\pm 1=1+q.$$Since $p\geq 2$ and $q\geq 3$, it is easy to see that these equations have no solutions.
\end{counterexample}

We proceed to catalog the almost/quasi-/Leinster groups in different classes of groups.
\section{Nilpotent groups}
The first result allows us to see what the almost Leinster groups are.
\begin{proposition}[\cite{Medts}]\label{prop:almost}
Let $G$ be a nilpotent group. The following are equivalent:
\begin{enumerate}[(1)]
\item
$\delta(G)\leq 2$
\item 
The group $G$ is cyclic and $|G|$ is perfect or deficient.
\end{enumerate}
\end{proposition}

\begin{proof}
The direct ((1)$\Rightarrow$(2)) is Proposition 3.2 from (\cite{Medts}). The converse follows from \eqref{D}.
\end{proof}
We can restate Proposition \ref{prop:almost}, using the new terminology.
\begin{corollary}
Let $G$ be a nilpotent group. The following are  equivalent:
\begin{enumerate}[(1)]
\item 
The group $G$ is almost/Leinster.
\item
The group $G$ is cyclic and $|G|$ is almost/perfect.
\end{enumerate}
\end{corollary}

The following proposition characterizes quasi-Leinster nilpotent groups.
\begin{proposition}
Let $G$ be a nilpotent group. The following are  equivalent:
\begin{enumerate}[(1)]
\item 
The group $G$ is quasi-Leinster.
\item
The group $G$ is cyclic and $|G|$ is quasi-perfect.
\end{enumerate}
\end{proposition}

\begin{proof}
The converse statement follows from \eqref{D}. We prove the direct.

Since $G$ is nilpotent, it can be written as the direct product of its Sylow $p$-subgroups:
\[G=G_1\times G_2\times \dots\times G_k \]
where $|G_i|=p_i^{n_i},$ for all $i\in\{1,2,\dots,k\}$. Since the function $D$ is multiplicative, it follows that:
\[ D(G)=\prod_{i=1}^k D(G_i).\]

Thus, we get \[\prod_{i=1}^k D(G_i)=2\prod_{i=1}^k |G_i|+1.\]

We assume, by contradiction, that $G_1$ is not cyclic. It follows that $n_1\geq 2$. In addition, the final theorem from \cite{Miler} states that the number of subgroups of index $p_1$ is congruent to $1+p_1\pmod{p_1^2}$; in particular, it is greater or equal to $1+p_1$. These subgroups are normal, so $D(G)\geq 1+p_1^{\alpha_1}+(p_1+1)p_1^{\alpha_1-1}.$ It follows that:
$$\delta(G_1)\geq \frac{1+p_1^{n_1}+(p_1+1)p_1^{n_1-1}}{p_1^{n_1}}=2+\frac{1}{p_1}+\frac{1}{p_1^{n_1}}\,.$$
Since $G$ is quasi-Leinster and $\delta$ is multiplicative, it follows that 
$$\begin{aligned}
2+\frac{1}{p_1^{n_1}\dots p_k^{n_k}}=\delta(G)&= \delta(G_1)\dots\delta(G_k)\\&\geq \left(2+\frac{1}{p_1}+\frac{1}{p_1^{n_1}}\right)\left(1+\frac{1}{p_2^{n_2}}\right)\dots\left(1+\frac{1}{p_k^{n_k}}\right)\\
&>2+\frac{1}{p_1^{n_1}\dots p_k^{n_k}}\,,
\end{aligned}$$a contradiction. Thus $G_1$ is cyclic. Analogously, all $G_i$, $i=2,\dots,k$, are cyclic and therefore $G$ is cyclic.
\end{proof}

Now, we study different classes of finite non-nilpotent groups.

\section{ZM-groups}

\begin{definition}[\cite{ZM}, Pages 144, 145]
A \textit{ZM-group} (Zassenhaus metacyclic) is a finite group with all Sylow subgroups cyclic.
\end{definition}

We recall some properties from literature for these groups.

\begin{proposition}[\cite{Huppert}]\label{prop:ZM}
The presentation of ZM-groups is given by:
\begin{equation}\label{eq:ZM}
ZM(m,n,r)=\langle a, b\mid a^m=b^n=1, b^{-1}ab=a^r\rangle,
\end{equation}
where the triplet $(m,n,r)\in(\mathbb{N}^*)^3$ satisfies the following conditions:
\begin{equation}\label{eq:mnr}
\begin{cases}
\gcd(m,n)=\gcd(m,r-1)=1,\\
r^n\equiv 1 \pmod m
\end{cases}
\end{equation}
\end{proposition}

\begin{lemma}[\cite{Tarna}]\label{ZMmnorder}
The order of the group $ZM(m,n,r)$ is $mn$. 
\end{lemma}

\begin{proposition}\label{qL}
There exist quasi-Leinster ZM-groups.
\end{proposition}
\begin{proof}
According to \cite[Theorem 5.7]{Medts}, the value of $D(ZM(m,n,r))$ is given by:
\begin{equation}\label{eq:sigmaZMN}
D(ZM(m,n,r))=mn\sum_{n_1|n}\frac{\delta(\gcd(m,r^{n_1}-1))}{n_1}.
\end{equation}
Assume that $m$ is prime, and let $d$ be the order of $r$ modulo $m$. According to \cite[Corollary 5.9.]{Medts}, equation \eqref{eq:sigmaZMN} becomes:
\begin{equation}
D(ZM(m,n,r))=mD(n)+D\left(\frac{n}{d}\right).
\end{equation}
Suppose further that $n=d$. It follows that:
\begin{equation}\label{md}
D(ZM(m,n,r))=mD(n)+1.
\end{equation} 
If we assume that $ZM(m,n,r)$ is quasi-Leinster, it follows from Lemma \ref{ZMmnorder} that: $$D(ZM(m,n,r))=2mn+1,$$ which, using \eqref{md}, implies $D(n)=2n$, i.e. $n$ is perfect.
On the other hand, it follows from \eqref{eq:mnr}, that 
\begin{equation}\label{r}
r\not\equiv 1\pmod{m}\text{ and }r\not\equiv 0\pmod{m}.
\end{equation} since $m$ is prime, it follows that the order of $r$ modulo $m$ is $m-1$. Since we have assumed that the order of $r$ is $n$, $m-1=n$, or, equivalently $m=n+1$, which has solutions, according to Example \ref{perfect+1}. Any $r\in\{2,\dots,m-1\}$ satisfies the conditions.
\end{proof}
\begin{example}
Using the assumptions in the previous proof and \ref{perfect+1}, we get $4$ known values for $n$, and the corresponding values for $m$. For each of these, all $r\in\{2,\dots,m-1\}$ give an example of quasi-Leinster ZM group. For each $m$, we write just one, for $r=3$: $(m,n,r)=(7,6,3)$, $(m,n,r)=(29,28,3)$, $(m,n,r)=(33550337, 33550336, 3)$, $(m,n,r)=(137438691329, 137438691328, 3)$.    
\end{example}

\section{Affine groups}

\begin{definition}[\cite{thesis}, Chapter 2]
Given $q=p^k$ with $p$ prime and $k\in\mathbb{N}^*$, we consider the affine group 
$$\Aff(\mathbb{F}_q)=\mathbb{F}_q\rtimes \mathbb{F}_q^*=\left\{\begin{pmatrix}
a & b\\
0 & 1
\end{pmatrix}: a\in \mathbb{F}_q^*, b\in \mathbb{F}_q\right\},$$
with matrix multiplication as the group operation.
\end{definition}

The normal subgroups of $\Aff(\mathbb{F}_q)$ are described in the following proposition.

\begin{proposition}[\cite{Conrad:affine}]\label{prop:affine}
There is a bijection between the subgroups of $\mathbb{F}_q^*$ and the non-trivial normal subgroups of $\Aff(\mathbb{F}_q)$, namely
$$H\leq \mathbb{F}_q^*\mapsto N_H=\begin{pmatrix}
H & \mathbb{F}_q\\
0 & 1
\end{pmatrix}=\left\{\begin{pmatrix}
a & b\\
0 & 1
\end{pmatrix}: a\in H, b\in \mathbb{F}_q\right\}\lhd \Aff(\mathbb{F}_q).$$
\end{proposition}
\begin{proposition}
\begin{enumerate}[(1)]
\item
There are no Leinster affine groups.
\item
The cyclic group of order $2$ is the only quasi-Leinster affine group.
\item
An affine group $\Aff(F_q)$ is almost Leinster if and only if $q$ is prime and $q-1$ is perfect.
\end{enumerate}
\end{proposition}
\begin{proof}
Since $\mathbb{F}_q^*$ is cyclic, it follows from Proposition \ref{prop:affine} that: 
$$N(\Aff(\mathbb{F}_q))=\{1\}\cup \{N_H|H\leq \mathbb{F}_q^*\}.$$Then $$D(\Aff(\mathbb{F}_q))=1+\sum_{H\leq \mathbb{F}_q^*}|N_H|=1+q\sum_{H\leq \mathbb{F}_q^*}|H|=1+qD(q-1).$$We obtain that:
\begin{enumerate}[(1)]
\item $\Aff(\mathbb{F}_q)$ is Leinster if and only if $1+qD(q-1)=2q(q-1)$. This yields $q|1$, which is absurd.
\item $\Aff(\mathbb{F}_q)$ is quasi-Leinster if and only if $1+qD(q-1)=2q(q-1)-1$. It follows that $ q|2$, thus $q=2$, i.e. $\Aff(\mathbb{F}_q)\cong C_2$.
\item $\Aff(\mathbb{F}_q)$ is almost Leinster if and only if $1+qD(q-1)=2q(q-1)+1$. This implies $D(q-1)=2(q-1)$, i.e. $q-1$ is perfect. According to Corollary \ref{k=1}, $q=p$.
\end{enumerate}
\end{proof}
\begin{example}
Using Example \ref{perfect+1}, we get $4$ examples of almost Leinster affine groups: $\Aff(\mathbb{F}_7)$, $\Aff(\mathbb{F}_{29})$, $\Aff(\mathbb{F}_{33550337})$, and $\Aff(\mathbb{F}_{137438691329})$.
\end{example}
\section{Dihedral groups}

\begin{definition}[\cite{WR}, Page 188]
Given $n\in\mathbb{N}^*\setminus\{1\}$, the corresponding \textit{dihedral group} is given by 
$$D_{2n}=\langle r, s | r^n=s^2=e,srs^{-1}=r^{-1}\rangle.$$
\end{definition}

The normal subgroups of $D_{2n}$ are described in the following result.

\begin{lemma}[\cite{Conrad2}]\label{NDn}
The normal subgroups of $D_{2n}$ are given by:
\begin{numcases}{N(D_{2n})=}\{D_{2n}\}\cup\{\langle r^d\rangle: d|n\}, & n\text{ odd} \label{odd}\\
\{D_{2n}\}\cup\{\langle r^d\rangle: d|n\}\cup \{\langle r^2, s\rangle, \langle r^2,rs\rangle\}, & n\text{ even} \label{even}.
\end{numcases}
In particular, there is at most one normal subgroup per index in $D_{2n}$ except for three normal subgroups $\langle r\rangle$, $\langle r^2, s\rangle$, $\langle r^2,rs\rangle$ of index $2$ when $n$ is even.
\end{lemma}

We are now able to study when $D_{2n}$ is a quasi-/almost/Leinster group. 

\begin{proposition}\label{dihed}
A dihedral group $D_{2n}$ is quasi-/almost/Leinster if and only if $n$ is odd and quasi-/almost/Leinster.
\end{proposition}

\begin{proof}
If we assume that $D_{2n}$ is almost/quasi-/Leinster, we get 
\begin{equation}\label{dihl}
D(D_{2n})=4n+a, a\in\{-1,0,1\}.
\end{equation}
We compute the value for $D(D_{2n})$, distinguishing two cases based on the parity of $n$:
\begin{description}
\item[$n$ odd.]
Using \eqref{odd} from Lemma \ref{NDn}, together with \eqref{dihl}, we get 
\begin{equation}\label{d2nodd}
2\cdot 2n+a=2n+D(n), a\in\{-1,0,1\},    
\end{equation}
which is equivalent to $D(n)=2n\pm a$, i.e. $n$ being almost/quasi-/perfect. 
\item[$n$ even.]
Again, using \eqref{dihl} and Lemma \ref{NDn}, we get: 
$4n+D(n)=4n+ a, a\in\{-1,0,1\}$, which has no solution.
\end{description}
\end{proof}
\begin{remark}
Proposition \ref{dihed} recovers the result for Leinster dihedral groups from \cite{Leinster}. 
\end{remark}
\section{Dicyclic groups}
\begin{definition}[\cite{art}, Page 128]
Given $n\in\mathbb{N}^*\setminus\{1\}$, the corresponding \textit{dicyclic group} is given by:
$$\Dic_{4n}=\langle a,x|a^{2n}=1, x^2=a^n, ax=xa^{-1}\rangle.$$
\end{definition}
The normal subgroups of the dicyclic groups can be described in a similar way to those of the dihedral groups.
\begin{lemma}[\cite{art}, Lemma 3.3]\label{ndic}
The normal subgroups of $\Dic_{4n}$ are given by:
\begin{numcases}
{N(\Dic_{4n})=}
{\Dic_{4n}}\cup\{\langle a^d\rangle: d|2n\}, & n\text{ odd} \label{dodd}\\
{\Dic_{4n}}\cup\{\langle a^d\rangle: d|2n\}\cup \{\langle a^2, ax\rangle, \langle a^2, a^2x\rangle, & n\text{ even} \label{deven}.
\end{numcases}
\end{lemma}

This allows us to prove the following proposition.

\begin{proposition}
A dicyclic group $\Dic_{4n}$ is quasi-/almost/Leinster if and only if $n$ is odd, and $2n$ quasi-/almost/Leinster.
\end{proposition}

\begin{proof}
If $\Dic_{4n}$ is almost/quasi-/Leinster, we get 
\begin{equation}\label{dicn}
D(\Dic_{4n})=8n + a, a\in\{-1,0,1\}
\end{equation}
Using lemma \ref{ndic}, we compute $D(\Dic_{4n})$ based on the parity of $n$.
\begin{description}
\item[$n$ even.]
In this case, it follows from \eqref{deven}, that $D(\Dic_{4n})=D(2n)+8n$. Together with \eqref{dicn}, we get $D(2n)+8n=8n + a, a\in\{-1,0,1\}$, which has no solution.
\item[$n$ odd.]
In this case $D(\Dic_{4n})=D(2n)+4n$. It follows that $D(2n)+4n=8n+a, a\in\{-1,0,1\}$, i.e.: 
\begin{equation}\label{eq:a}
D(2n)=4n + a, a\in\{-1,0,1\}.
\end{equation}
\begin{enumerate}[(1)]
\item
If $D(2n)=4n-1$, then $2n$ is almost perfect.
\item
If $D(2n)=4n$, then $2n$ is perfect.
\item
If $D(2n)=4n+1$, then $2n$ is quasi-perfect.
\end{enumerate}
\end{description}
\end{proof}
\begin{remark}
In the context of the previous proof, when $n$ is odd, equation \eqref{eq:a} 
has no known solutions for $a\in\{-1,1\}$. \\
Indeed, there are no known almost perfect numbers of the form $4k+2$, $k\in\mathbb{N}^*$, so there are no known solutions for $a=-1$.\\
Similarly, there are no known solution for $a=1$, since there are no known quasi-perfect numbers.\\
For $a=0$, or equivalently, $2n$ perfect, there is just one solution, $n=3$, which yields the unique Leinster group $\Dic_{12}$.
\end{remark}
\begin{remark}
Sometimes, the dicyclic groups are called generalized quaternion groups, see \cite[Page 252]{WR}.\\
Our results above recover the result for Leinster generalized quaternion groups from \cite[Proposition 2.1.]{Baishya1}.
\end{remark}

\section{Conclusions}
We introduced quasi-/almost Leinster groups, which are analogues to quasi-/almost perfect numbers. \\
We proved that the only nilpotent almost and quasi-Leinster groups are the cyclic groups of almost and quasi-perfect order.\\
In addition, we studied a few well-known classes of non-nilpotent groups (ZM-groups, affine groups, dihedral groups and dicyclic groups). \\
We found examples of almost and quasi-Leinster groups for the first two. 
We proved that the existence of almost/quasi-/Leinster dihedral groups $D_{2n}$ groups is equivalent to $n$ being odd, and $2n$ being almost/quasi-/perfect.\\
Similarly, for dicyclic groups, we proved that the existence of an almost/quasi-/Leinster group $\Dic_{4n}$ is equivalent to $n$ being odd and $2n$ being almost/quasi-/perfect.
\begin{acknowledgments}
We thank the anonymous referee for the valuable comments which improved the first versions of the paper.
\end{acknowledgments}

 \begin{flushleft}
Iulia - C\u at\u alina PLE\c SCA\\
Faculty of Mathematics,\\
"Alexandru Ioan Cuza" University of Ia\c si,\\
Carol I Boulevard, Nr. 11, 700506 Ia\c si, Rom\^ania. \\
Email: iulia.plesca@uaic.ro, dankemath@yahoo.com\\
ORCID: 0000-0001-7140-844X

Marius T\u ARN\u AUCEANU\\
Faculty of Mathematics,\\
"Alexandru Ioan Cuza" University of Ia\c si,\\
Carol I Boulevard, Nr. 11, 700506 Ia\c si, Rom\^ania. \\
Email: tarnauc@uaic.ro\\
ORCID: 0000-0003-0368-6821
\end{flushleft}

\begin{thebibliography}{99}
\bibitem{Baishya1} S.J. Baishya, \textit{Revisiting the Leinster groups}, C.R. Math. {\bf 352} (2014), 1-6.
\bibitem{Baishya2} S.J. Baishya, \textit{On Leinster groups of order $pqrs$}, arXiv:1911.04829. 
\bibitem{Calhoun} W.C. Calhoun, \textit{Counting subgroups of some finite groups}, Amer. Math. Monthly {\bf 94} (1987), 54-59.
\bibitem{Conrad:affine} K. Conrad, \textit{Normal Subgroups of Aff(F)}, \url{https://kconrad.math.uconn.edu/blurbs/grouptheory/affinenormal.pdf}.
\bibitem{Conrad2} K. Conrad, \textit{Dihedral groups ii}, \url{http://www.math.uconn.edu/ kconrad/blurbs/grouptheory/dihedral2.pdf}.
\bibitem{Guy} R. Guy, \textit{Unsolved Problems in Number Theory}, third edition, Springer Verlag, 2004.
\bibitem{Huppert} B. Huppert, \textit{Endliche Gruppen}, I, Springer Verlag, Berlin, 1967.
\bibitem{Isaacs} I.M. Isaacs, \textit{Finite Group Theory}, Amer. Math. Soc., Providence, R.I., 2008.
\bibitem{Leinster} T. Leinster, \textit{Perfect numbers and groups}, Eureka {\bf 55} (2001), 17–27.
\bibitem{thesis} D. A. Lingenbrink Jr., \textit{A New Subgroup Chain for the Finite Affine Group}, (2014), HMC Senior Theses.
\bibitem{Medts} T. De Medts and A. Mar\'{o}ti, \textit{Perfect numbers and finite groups}, Rend. Semin. Mat. Univ. Padova {\bf 129} (2013), 17–34.
\bibitem{Miler} G. A. Miller, \textit{Form of the number of subgroups of prime power groups}, Bull. Amer. Math. Soc. 26 (1919), 66-72.
\bibitem{Nieuwveld} J. Nieuwveld, \textit{A note on Leinster groups}, \url{https://www.math.ru.nl/~bosma/Students/JorisNieuwveld/A_note_on_Leinster_groups.pdf}.
\bibitem{art} H. Saydi, Normal supercharacter theories of the dicyclic groups, International Electronic Journal of Algebra, \textbf{37} (2025), 125-139.
\bibitem{WR} W. R. Scott - \textit{Group Theory}, Dover Publications, New York, 1987.
\bibitem{Tarna} M. Tărnăuceanu, \textit{The normal subgroup structure of ZM-groups}, Ann. Mat. Pura Appl. {\bf 193} (2014), 1085–1088. 
\bibitem{ZM} H. J. Zassenhaus, \textit{The theory of groups}, Chelsea Publishing Company, 1949.
\bibitem{GAP}The GAP Group, GAP – groups, algorithms, and programming, version 4.11.0, (2020), \url{https://www.gap-system.org}.
\bibitem{Mersenne} The Great Internet Mersenne Prime Search, \url{https://www.mersenne.org/}.
\bibitem{post} MathOverflow, \textit{Is there an odd-order group whose order is the sum of the orders of the proper
normal subgroups?}, (2011), \url{ https://mathoverflow.net/questions/54851}.
\end{thebibliography}
\end{document}